\documentclass[11pt]{article}
\usepackage{amsmath,amsthm,amssymb,graphicx}
\usepackage[margin=1in]{geometry}
\usepackage{amsfonts}
\usepackage{hyperref}
\usepackage{tikz}
\usepackage{booktabs}
\usepackage[all]{xy}
\usepackage{braket}
\newtheorem{lemma}{Lemma}

\newtheorem{theorem}{Theorem}
\newtheorem{definition}{Definition}
\newtheorem{corollary}{Corollary}

\newtheorem{proposition}{Proposition}
\newtheorem{remark}{Remark}

\usepackage{algorithm}
\usepackage{algorithmic}
\theoremstyle{plain}

\usepackage{subcaption}
\usepackage{placeins} 

\title{The Antisymmetric Line Graph}

\author{Hartosh Singh Bal}
\date{}

\begin{document}

\maketitle
\renewcommand{\thefootnote}{}
\renewcommand{\thefootnote}{\arabic{footnote}}

\begin{abstract}
Let $G$ be a finite simple graph with oriented incidence matrix $D$. The signed
graph on edge set $E(G)$ with adjacency matrix
\[
A_{\mathcal A(G)}=D^{\mathsf T}D-2I
\]
is classical in the signed-line-graph literature. In this paper we study its
canonical switching class as a source of invariants of the underlying unsigned
graph.

We prove that the switching class of $\mathcal A(G)$ determines $G$ up to
isomorphism modulo isolated vertices, and we relate the frustration index
$\ell(\mathcal A(G))$ to classical bipartization parameters. In particular, we
show
\[
\operatorname{def}(G)\le \ell(\mathcal A(G))\le (\Delta(G)-1)\operatorname{def}(G),
\]
and, for cubic graphs,
\[
\ell(\mathcal A(G))=2\,\operatorname{oct}(G).
\]

We then prove the exact optimization identity
\[
\ell(\mathcal A(G))
=
\frac14\sum_{v\in V(G)} d(v)^2
-\frac14\max_{x\in\{\pm1\}^{E(G)}}\|Dx\|^2,
\]
so $\ell(\mathcal A(G))$ is exactly a Boolean edge-space Laplacian optimization
problem. This yields a spectral lower bound in terms of the largest Laplacian
eigenvalue, a cubic spectral lower bound on odd cycle transversal, and explicit
family-level comparisons showing that the spectral and defect bounds govern
different regimes: on odd cycles the spectral bound is asymptotically vacuous,
while on complete multipartite graphs it already captures exactly $3/4$ of the
true value of $\ell(\mathcal A(G))$.

Thus the paper uses a classical signed line graph in a new way: as a
source of combinatorial invariants of ordinary graphs, especially through
frustration and odd-cycle-transversal phenomena.
\end{abstract}
\section{Introduction}

In this paper we study the \emph{antisymmetric line graph} (ALG), the signed graph
on edge set $E(G)$ with adjacency matrix
\[
A_{\mathcal A(G)}=D^{\mathsf T}D-2I,
\]
where $D$ is an oriented incidence matrix of $G$. This object is classical in
the signed-line-graph literature: up to the prevailing sign convention it is the
signed or spectral line graph of an oriented graph
\cite{GerminaHameedZaslavsky2011,Zaslavsky2010Matrices,BelardoPisanskiStanicZaslavsky2023}.
Its switching-invariance under edge reorientation, its cycle-parity behavior,
and its balance-theoretic connection with bipartiteness are also part of that
literature. We retain the name \emph{antisymmetric line graph} because this
signed graph arises naturally as the antisymmetric sector of the doubled lift
$\mathrm{HL}'_2(G)$ in \cite{BalHL}.

The purpose of the present paper is different. Rather than treating
$\mathcal A(G)$ as a signed-graph construction in its own right, we use its
canonical switching class as a source of invariants of the underlying unsigned
graph $G$. In this direction, the main new results concern frustration,
odd-cycle-transversal parameters, and spectral lower bounds derived from the same
incidence formula.

Conceptually, ALG sits between the base graph and the ordinary line graph.
Forgetting signs recovers the line graph $L(G)$, but the signed structure
remembers orientation-sensitive local incidence information. This extra
structure is enough, for example, to distinguish the Whitney pair
$\{K_3,K_{1,3}\}$ and hence recover the base graph from the switching class up to
isolated vertices. At the same time, the identity
\[
A_{\mathcal A(G)}+2I=D^{\mathsf T}D
\]
places the Laplacian directly on edge space and explains why the frustration
index of $\mathcal A(G)$ admits an exact Boolean-Laplacian formulation.

The matrix
\[
D^{\mathsf T}D-2I
\]
also has a classical spectral history: it is the Gram matrix arising from the
root-system representation of line graphs, and its least-eigenvalue bound
$\ge -2$ goes back to Cameron, Goethals, Seidel, and Shult \cite{CGSS}. What is
specific here is not the matrix alone, but the unsigned-graph information that
can be extracted from the canonical signed switching class it defines.

The antisymmetric line graph also arises naturally from symmetric lift theory:
in \cite{BalHL} the doubled lift $\mathrm{HL}'_2(G)$ was shown to split into
symmetric and antisymmetric sectors, the former governed by $L(G)$ and the latter
by $\mathcal A(G)$. The present paper isolates that antisymmetric sector and
studies its intrinsic graph-theoretic consequences.

\paragraph{Main results.}
Our strongest results are frustration-theoretic. We prove an orientation formula
for $\ell(\mathcal A(G))$, compare it with Max-Cut defect, and establish the
exact cubic identity
\[
\ell(\mathcal A(G))=2\,\mathrm{oct}(G).
\]
We then show that
\[
\ell(\mathcal A(G))
=
\frac14\sum_{v\in V(G)} d(v)^2
-
\frac14\max_{x\in\{\pm1\}^{E(G)}} \|Dx\|^2,
\]
so the frustration index is exactly a Boolean optimization of the edge-space
Laplacian quadratic form. This yields a spectral lower bound in terms of the
largest Laplacian eigenvalue and, in the cubic case, a spectral lower bound on
odd cycle transversal. We also evaluate $\ell(\mathcal A(G))$ exactly on complete
multipartite graphs, where the single-eigenvalue spectral lower bound is already
$3/4$-sharp.

Along the way we record several structural properties of the canonical switching
class of $\mathcal A(G)$, including reconstruction up to isolated vertices,
cycle-parity constraints on induced signed cycles, and the compact edge-space
Laplacian geometry carried by the same matrix identity.

\medskip
\noindent\textbf{What is new in this paper.}
The main new contributions are:
\begin{enumerate}\itemsep2pt
\item[(i)] the use of the canonical signed line graph $\mathcal A(G)$ as a source
of invariants of the underlying unsigned graph, rather than as an object of
signed-spectral theory in its own right;
\item[(ii)] the reconstruction theorem from the switching class of $\mathcal A(G)$,
excluding Whitney's exceptional ambiguity;
\item[(iii)] the orientation characterization of $\ell(\mathcal A(G))$, the
comparison with Max-Cut defect, and in particular the cubic exactness theorem
\[
\ell(\mathcal A(G))=2\,\mathrm{oct}(G);
\]
\item[(iv)] the exact Boolean-Laplacian reformulation of $\ell(\mathcal A(G))$,
together with the resulting spectral lower bound in terms of the largest
ordinary Laplacian eigenvalue and the cubic spectral lower bound for odd cycle
transversal;
\item[(v)] explicit model families showing that the defect and spectral lower
bounds govern genuinely different regimes, including exact evaluation on complete
multipartite graphs where the spectral lower bound captures exactly $3/4$ of the
true value;
\item[(vi)] small-scale computational evidence illustrating both separating power
and limitations of low-order and spectral signed invariants of $\mathcal A(G)$.
\end{enumerate}

The paper is organized as follows.
Section~\ref{sec:preliminaries} reviews line graphs, signed graphs, and
incidence matrices.
Section~\ref{sec:construction} gives the intrinsic construction of
$\mathcal A(G)$ and situates it relative to the signed-line-graph literature.
Section~\ref{sec:laplacian-geometry} records the compact edge-space Laplacian
interpretation of the same operator.
Section~\ref{sec:signed-class} develops the intrinsic signed-graph theory.
Section~\ref{sec:triangle} studies the signed triangle imbalance.
Section~\ref{sec:empirical} gives brief computational illustrations.
Section~\ref{sec:frustration-maxcut} contains the frustration, Max-Cut, and
cubic exactness results.
Section~\ref{sec:laplacian-optimization} reformulates $\ell(\mathcal A(G))$ as
a Boolean Laplacian optimization problem and derives the spectral bounds.
Section~\ref{sec:outlook} concludes.

\section{Preliminaries}
\label{sec:preliminaries}

In this section we fix notation and recall standard definitions concerning line
graphs and signed graphs.
All graphs in this paper are finite, simple, and undirected unless stated
otherwise.

\subsection{Graphs and Line Graphs}

Let $G=(V,E)$ be a graph.
The \emph{line graph} $L(G)$ is the graph whose vertex set is $E(G)$, with two
vertices adjacent if and only if the corresponding edges of $G$ share a common
endpoint.
We write $A_{L(G)}$ for the adjacency matrix of $L(G)$.

The line graph construction is functorial with respect to graph isomorphisms, but
it is not injective.
Whitney’s classical theorem characterizes the failure of injectivity.

\begin{theorem}[Whitney {\cite{Whitney1932}}]
\label{thm:whitney}
Let $G$ and $H$ be connected graphs.
If $L(G)\cong L(H)$, then either $G\cong H$, or $\{G,H\}=\{K_3,K_{1,3}\}$.
\end{theorem}

Since line graphs are defined on edges, they do not detect isolated vertices.
Accordingly, reconstruction results involving $L(G)$ are always understood up to
the addition or removal of isolated vertices.

\subsection{Signed Graphs}

A \emph{signed graph} is a pair $(\Gamma,\sigma)$, where $\Gamma=(V,E)$ is a graph
and $\sigma:E\to\{+1,-1\}$ assigns a sign to each edge.
Equivalently, a signed graph may be specified by a symmetric matrix
$A=(a_{ij})_{i,j\in V}$ with $a_{ij}\in\{0,+1,-1\}$, where $|a_{ij}|$ records
adjacency and $\operatorname{sign}(a_{ij})$ records the edge sign.

A fundamental notion in signed graph theory is \emph{switching}.
Given a signed graph $(\Gamma,\sigma)$ and a vertex $v\in V$, switching at $v$
means reversing the signs of all edges incident to $v$.

\begin{definition}[Switching class]
Two signed graphs $\Sigma_1$ and $\Sigma_2$ on the same underlying unsigned graph are
\emph{switching equivalent} if $\Sigma_2$ can be obtained from $\Sigma_1$ by a sequence of switchings.
The corresponding equivalence class is called the \emph{switching class} of $\Sigma$.
\end{definition}

Switching equivalence preserves all intrinsic signed-graph properties, including
cycle balance, frustration index, and signed spectra.
Accordingly, signed graphs are typically studied up to switching equivalence.

\subsection{Cycles and Balance}

Let $(\Gamma,\sigma)$ be a signed graph.
For a cycle $C\subseteq \Gamma$, the \emph{sign of $C$} is defined as the product
of the signs of its edges.
A cycle is called \emph{positive} if this product is $+1$ and \emph{negative} if
it is $-1$.
A signed graph is said to be \emph{balanced} if all of its cycles are positive.

In the present work, we will be particularly interested in signed triangles in
the antisymmetric line graph.
The imbalance between positive and negative triangles will give rise to a
canonical invariant of ordinary graphs, introduced in later sections.

\subsection{Incidence Matrices}

Let $G=(V,E)$ be a graph.
Fix an arbitrary orientation of each edge.
The (oriented) incidence matrix $D$ of $G$ is the $|V|\times|E|$ matrix defined by
\[
D_{ve} =
\begin{cases}
+1 & \text{if $v$ is the head of $e$,}\\
-1 & \text{if $v$ is the tail of $e$,}\\
0  & \text{otherwise.}
\end{cases}
\]

The matrix $D^{\mathsf T}D$ encodes adjacency relations among edges together with their
relative orientations.
Changing the orientation of a single edge corresponds to multiplying the
corresponding column of $D$ by $-1$.
As we shall see, this operation induces a switching transformation at the
corresponding vertex of the line graph.

This observation underlies the canonical nature of the antisymmetric line graph
construction developed in Section~\ref{sec:construction}.

\section{Construction of the Antisymmetric Line Graph}
\label{sec:construction}

The signed graph studied here can be obtained abstractly as the antisymmetric
sector of the symmetric lift $\mathrm{HL}'_2(G)$ introduced in \cite{BalHL}.
From a different direction, however, the same signed graph is classical in the
signed-line-graph literature: for an oriented graph with incidence matrix $D$,
the matrix
\[
D^{\mathsf T}D-2I
\]
is the signed or spectral line graph, up to the standard sign convention
\cite{GerminaHameedZaslavsky2011,Zaslavsky2010Matrices,BelardoPisanskiStanicZaslavsky2023}.

For the purposes of the present paper we work intrinsically from the incidence
matrix and retain the name \emph{antisymmetric line graph} because it is this
object's role in the doubled-lift decomposition that motivates our study. The
novelty here lies not in first introducing the signed graph itself, but in
developing new unsigned-graph consequences from its canonical switching class,
especially those involving frustration, odd-cycle transversals, and spectral
lower bounds.

\subsection{Canonical Orientation and Incidence}

Let $G=(V,E)$ be a finite simple graph.
Although $G$ is undirected, we fix an arbitrary reference orientation of its
edges for the purpose of computation.
This choice plays no structural role and will be shown to affect the
construction only up to switching equivalence.

Let $D$ be the $|V|\times|E|$ oriented incidence matrix of $G$, as defined in
Section~\ref{sec:preliminaries}.
The matrix $D^TD$ encodes both adjacency relations among edges and their relative
orientations.

The key point is that the same matrix $D^{\mathsf T}D$ simultaneously records
which pairs of edges meet and, when they do meet, whether the two incidences are
coherent or conflicting. This leads to the following signed refinement of the
line graph.

\subsection{Definition via Signed Adjacency}

Recall that the adjacency matrix of the ordinary line graph can be written as
\[
A_{L(G)} = |D^{\mathsf T} D| - 2I,
\]
where $|\cdot|$ denotes entrywise absolute value and $I$ is the identity matrix.
In particular, $|D^{\mathsf T} D|$ records which pairs of edges share a common endpoint.

\begin{definition}[Antisymmetric Line Graph]
\label{def:ALG}
The \emph{antisymmetric line graph} $\mathcal{A}(G)$ is the signed graph with
vertex set $E(G)$ and signed adjacency matrix
\begin{equation}
\label{eq:ALG-def}
A_{\mathcal{A}(G)} \;:=\; D^{\mathsf T} D - 2I .
\end{equation}
Equivalently, since $A_{L(G)}$ has $0/1$ off--diagonal entries and zero diagonal,
one may write
\[
A_{\mathcal{A}(G)} \;=\; (D^{\mathsf T} D)\circ A_{L(G)},
\]
where $\circ$ denotes the Hadamard (entrywise) product.
\end{definition}

\begin{lemma}[ALG refines the line graph]
\label{lem:ALG-refines-LG}
Let $G$ be a finite graph and let $\mathcal A(G)$ be its antisymmetric line graph.
Then
\[
|A_{\mathcal A(G)}| = A_{L(G)}.
\]
Equivalently, the underlying unsigned graph of $\mathcal A(G)$ is exactly the
line graph $L(G)$.
\end{lemma}

\begin{proof}
By Definition~\ref{def:ALG},
\[
A_{\mathcal A(G)} = D^{\mathsf T}D - 2I.
\]
The matrix $D^{\mathsf T}D$ has diagonal entries $2$, and its off-diagonal entry
$(D^{\mathsf T}D)_{ef}$ is nonzero if and only if the distinct edges $e$ and $f$
share a common endpoint in $G$; in that case $(D^{\mathsf T}D)_{ef}\in\{\pm1\}$.
Subtracting $2I$ removes the diagonal and leaves the same $0/1$ support pattern
as the line graph adjacency matrix. Taking entrywise absolute values yields
$|A_{\mathcal A(G)}|=A_{L(G)}$.
\end{proof}

Equivalently, two distinct edges $e,f\in E(G)$ are adjacent in
$\mathcal{A}(G)$ if and only if they share a common endpoint in $G$.
The sign of this adjacency is determined as follows:
\begin{itemize}
  \item if $e$ and $f$ have the \emph{same incidence} at their common vertex
  (both enter or both leave), the corresponding edge of
  $\mathcal{A}(G)$ is assigned sign $+1$;
  \item if $e$ and $f$ are oriented \emph{transitively} through their common
  vertex (one enters and the other leaves), the corresponding edge is
  assigned sign $-1$.
\end{itemize}
Edges of $G$ that do not share a vertex are nonadjacent in
$\mathcal{A}(G)$.

\subsection{Canonical Nature and Switching Invariance}

Although the definition above depends on a chosen reference orientation of $G$,
the resulting signed graph is canonical in the sense of signed graph theory.

\begin{proposition}[Orientations and switching]
\label{prop:switching}
Let $\vec{G}$ and $\vec{G}'$ be two reference orientations of the same graph $G$.
If $\vec{G}'$ is obtained from $\vec{G}$ by reversing exactly the edges in a set
$F\subseteq E(G)$, then the corresponding antisymmetric line graphs
$\mathcal{A}(\vec{G})$ and $\mathcal{A}(\vec{G}')$ differ by switching at the
vertex set $F\subseteq V(L(G))=E(G)$.
In particular, all orientations of $G$ induce exactly one switching class of
signed line graphs.
\end{proposition}

\begin{proof}
Reversing the orientation of a single edge $e$ multiplies the corresponding
column of the oriented incidence matrix $D$ by $-1$. Hence in $D^{\mathsf T}D$
the $e$th row and $e$th column are multiplied by $-1$, while all other entries
are unchanged. This is exactly switching at the vertex $e$ of the signed graph
$\mathcal A(G)$, whose vertex set is $E(G)$. Reversing all edges in
$F\subseteq E(G)$ therefore induces switching at the vertex set $F$, and the
final statement follows.
\end{proof}

\begin{corollary}
\label{cor:switching-class-size}
Let $m=|E(G)|$, and let $c(L(G))$ denote the number of connected components of
the line graph $L(G)$. Then the map from orientations of $G$ to signed graphs in
the switching class of $\mathcal A(G)$ is surjective, each element of the
switching class is induced by exactly $2^{c(L(G))}$ orientations, and therefore
the switching class has cardinality
\[
2^{\,m-c(L(G))}.
\]
In particular, if $L(G)$ is connected, then the switching class of
$\mathcal A(G)$ has size $2^{m-1}$.
\end{corollary}

\begin{proof}
There are $2^m$ orientations of $G$, one for each choice of direction on each
edge. By Proposition~\ref{prop:switching}, two orientations induce the same
signed graph precisely when they differ by switching that is trivial on each
connected component of $L(G)$. On a connected component of a signed graph,
switching at all vertices leaves every edge sign unchanged, and these are the
only switching subsets with this property. Thus each induced signed graph has
exactly $2^{c(L(G))}$ orientation-preimages, one for each choice of whether or
not to reverse all edges in each component of $L(G)$. Dividing $2^m$ by
$2^{c(L(G))}$ gives the formula.
\end{proof}
The construction above is not ad hoc. In the symmetric lift $\mathrm{HL}'_2(G)$, the involution reversing edge orientations induces a decomposition into symmetric and antisymmetric subspaces
\cite{BalHL}.
The operator governing the antisymmetric subspace is precisely $D^{\mathsf T}D - 2I$,
restricted to edge adjacencies. Under this identification, the ordinary line graph $L(G)$ governs the symmetric sector of the lift, while the antisymmetric line graph $\mathcal{A}(G)$ governs the antisymmetric sector. The present work isolates $\mathcal{A}(G)$ as an object of independent interest.

\section{ALG as a canonical edge-space realization of the Laplacian}
\label{sec:laplacian-geometry}

The same incidence formula that defines the antisymmetric line graph also places
the Laplacian of $G$ directly on edge space.
This is not a second main theorem line of the paper, but a compact structural
observation that helps explain why $\mathcal A(G)$ is a natural object.

Throughout this section, let $G$ be a connected finite simple graph with
oriented incidence matrix $D$, Laplacian
\[
L:=DD^{\mathsf T},
\]
and signed adjacency matrix
\[
S:=A_{\mathcal A(G)}=D^{\mathsf T}D-2I.
\]

\begin{proposition}[Edge-space Laplacian realization]
\label{prop:edge-space-laplacian}
One has
\[
S+2I=D^{\mathsf T}D.
\]
Consequently:
\begin{enumerate}
\item the nonzero eigenvalues of $S+2I$ coincide with the nonzero eigenvalues of
$L$;
\item the edge space admits the orthogonal decomposition
\[
\mathbb R^{E(G)}=\operatorname{im}(D^{\mathsf T})\oplus\ker(D);
\]
\item the cycle space is exactly the distinguished $(-2)$-eigenspace:
\[
\ker(D)=\ker(S+2I).
\]
\end{enumerate}
\end{proposition}

\begin{proof}
The identity is immediate from Definition~\ref{def:ALG}. Since $D^{\mathsf T}D$
and $DD^{\mathsf T}$ have the same nonzero eigenvalues, (1) follows from standard
linear algebra. The orthogonal decomposition in (2) is the identity
\[
\operatorname{im}(D^{\mathsf T})^{\perp}=\ker(D).
\]
For (3), one has
\[
(S+2I)x=0
\Longleftrightarrow
D^{\mathsf T}Dx=0
\Longleftrightarrow
x^{\mathsf T}D^{\mathsf T}Dx=\|Dx\|^2=0
\Longleftrightarrow
Dx=0.
\]
\end{proof}

\begin{proposition}[Transport of Laplacian modes]
\label{prop:transported-modes}
Let
\[
0=\lambda_1<\lambda_2\le \cdots \le \lambda_n
\]
be the Laplacian eigenvalues of $G$, with orthonormal eigenvectors
$v_1,\dots,v_n$ satisfying $Lv_i=\lambda_i v_i$.
Then for each $i\ge 2$ the edge vector
\[
w_i:=D^{\mathsf T}v_i
\]
is a nonzero eigenvector of $S$ with eigenvalue $\lambda_i-2$:
\[
Sw_i=(\lambda_i-2)w_i.
\]
Moreover,
\[
\|w_i\|^2=\lambda_i,
\qquad
\langle w_i,w_j\rangle=\lambda_i\delta_{ij}\quad(i,j\ge 2),
\]
so the normalized vectors
\[
\left\{\frac{D^{\mathsf T}v_i}{\sqrt{\lambda_i}}:2\le i\le n\right\}
\]
form an orthonormal basis of $\operatorname{im}(D^{\mathsf T})$.
\end{proposition}

\begin{proof}
Since $v_1$ is constant on a connected graph, $D^{\mathsf T}v_1=0$. For $i\ge2$,
\[
S(D^{\mathsf T}v_i)
=(D^{\mathsf T}D-2I)D^{\mathsf T}v_i
=D^{\mathsf T}DD^{\mathsf T}v_i-2D^{\mathsf T}v_i
=D^{\mathsf T}Lv_i-2D^{\mathsf T}v_i
=(\lambda_i-2)D^{\mathsf T}v_i.
\]
Also,
\[
\|D^{\mathsf T}v_i\|^2=v_i^{\mathsf T}DD^{\mathsf T}v_i=v_i^{\mathsf T}Lv_i=\lambda_i,
\]
and similarly
\[
\langle D^{\mathsf T}v_i,D^{\mathsf T}v_j\rangle
=v_i^{\mathsf T}DD^{\mathsf T}v_j
=v_i^{\mathsf T}Lv_j
=\lambda_j v_i^{\mathsf T}v_j
=\lambda_i\delta_{ij}.
\]
\end{proof}

\begin{corollary}[The transported Fiedler mode]
\label{cor:fiedler-transport}
Let $v_2$ be a unit Fiedler vector of $G$, so that $Lv_2=\lambda_2v_2$.
Then
\[
w_2:=D^{\mathsf T}v_2
\]
is an eigenvector of $S$ with eigenvalue $\lambda_2-2$.
For every oriented edge $e=(x,y)$,
\[
w_2(e)=v_2(y)-v_2(x),
\]
so $|w_2(e)|$ records the edgewise variation of the Fiedler mode across $e$.
\end{corollary}

\begin{proof}
This is the case $i=2$ of Proposition~\ref{prop:transported-modes}, together with
the coordinate formula coming directly from the incidence column of an oriented
edge.
\end{proof}

\begin{corollary}[Matrix-tree identity on $\mathcal A(G)$]
\label{cor:matrix-tree}
If $G$ is connected on $n$ vertices, then
\[
\tau(G)=\frac1n\det\nolimits'(S+2I),
\]
where $\det'(M)$ denotes the product of the nonzero eigenvalues of $M$.
\end{corollary}

\begin{proof}
By the matrix-tree theorem,
\[
\tau(G)=\frac1n\prod_{i=2}^n\lambda_i.
\]
By Proposition~\ref{prop:edge-space-laplacian}, the nonzero eigenvalues of
$S+2I$ are exactly $\lambda_2,\dots,\lambda_n$.
\end{proof}

\begin{remark}\label{rem:laplacian-geometry-role}
The results of this section are elementary consequences of the shifted Gram
identity
\[
A_{\mathcal A(G)}+2I=D^{\mathsf T}D,
\]
but they are worth recording because they show that the same canonical signed
graph also carries the classical cut/cycle and transported Laplacian geometry of
edge space.
A fuller edge-flow treatment is to be developed separately.
\end{remark}

\section{Antisymmetric Line Graphs within Signed Graph Theory}
\label{sec:signed-class}

The antisymmetric line graph embeds ordinary graphs into the theory of signed
graphs in a highly constrained and canonical way.
In this section we clarify how antisymmetric line graphs fit within the broader
landscape of signed graph theory, and how they differ from arbitrary signings of
line graphs.

\subsection{Signed Line Graphs versus Antisymmetric Line Graphs}

Signed line graphs of signed graphs are classical objects in signed graph theory;
see, for instance, Zaslavsky's matrix formula
\[
A(\Lambda(\Sigma))=2I-H(\Sigma)^{\mathsf T}H(\Sigma)
\]
and the later signed-spectral treatment of the spectral line graph
\cite{Zaslavsky2010Matrices,GerminaHameedZaslavsky2011,BelardoPisanskiStanicZaslavsky2023}.
For an ordinary graph $G$ with a chosen orientation and incidence matrix $D$,
our matrix
\[
D^{\mathsf T}D-2I
\]
is exactly this classical signed-line-graph construction up to the prevailing
sign convention. We retain the name \emph{antisymmetric line graph} to emphasize
its role as the antisymmetric sector of the doubled lift in \cite{BalHL}.

The point of the present paper is therefore not to introduce a previously
unknown signed graph, but to use this canonical switching class as a source of
combinatorial invariants of the underlying unsigned graph. In particular, we
focus on frustration, odd-cycle-transversal phenomena, and spectral lower bounds
for these unsigned parameters, directions that do not appear to have been
pursued in the earlier signed-line-graph literature.

It is important to note that not every signing whose underlying
unsigned graph is a line graph arises as an antisymmetric line graph.
Indeed, arbitrary signings of $L(G)$ generally violate the local compatibility
conditions imposed by shared $2$--paths in the base graph.

\subsection{Rigidity of the ALG Signing}

The signing of $\mathcal{A}(G)$ is rigid in a strong sense.

\begin{proposition}[Canonical signing]\label{prop:unique-signing}
For every finite simple graph $G$, the signed graph $\mathcal A(G)$ is uniquely determined up to switching.
\end{proposition}

\begin{proof}
Choose any reference orientation of $G$, and let $D$ be the corresponding
oriented incidence matrix. By Definition~\ref{def:ALG}, the signed adjacency
matrix of $\mathcal A(G)$ is
\[
A_{\mathcal A(G)} = D^{\mathsf T}D - 2I.
\]
If a second reference orientation is obtained by reversing the edges in
$F\subseteq E(G)$, then its oriented incidence matrix is $D' = DQ$, where
$Q$ is the diagonal matrix indexed by $E(G)$ with entries
\[
Q_{ee} =
\begin{cases}
-1,& e\in F,\\
\phantom{-}1,& e\notin F.
\end{cases}
\]
Therefore
\[
A_{\mathcal A(\vec G')}
= (D')^{\mathsf T}D' - 2I
= Q(D^{\mathsf T}D - 2I)Q
= Q\,A_{\mathcal A(\vec G)}\,Q.
\]
Conjugation by a diagonal $\{\pm1\}$-matrix is precisely switching in signed
graph theory. Hence any two reference orientations of $G$ induce
switching-equivalent signed graphs, so $\mathcal A(G)$ is uniquely determined
up to switching.
\end{proof}

Thus, in the language of signed graph theory, $\mathcal A(G)$ is not an arbitrary
signing of a line graph but the canonical switching class associated to the
oriented incidence structure of $G$. This canonicality is classical in the
signed-line-graph setting \cite{Zaslavsky2010Matrices,BelardoPisanskiStanicZaslavsky2023};
what matters here is that this specific switching class supports new unsigned
graph invariants when one studies its frustration and related parameters.

\subsection{Forbidden Structure and Necessary Conditions}

Classical results characterize which graphs arise as line graphs via forbidden
induced subgraphs \cite{Beineke1970,West2001}.
Since the underlying unsigned graph of $\mathcal{A}(G)$ is always a line graph,
any signed graph arising as an antisymmetric line graph must, in particular,
satisfy all of these classical constraints.

Beyond the unsigned structure, the signing imposes additional restrictions.
In particular:
\begin{itemize}
  \item the sign of any triangle in $\mathcal{A}(G)$ is determined by whether the
  corresponding triple of edges in $G$ forms a cyclic or branching configuration;
  \item signed cycles must be compatible across overlapping $2$--paths;
  \item switching equivalence cannot eliminate signed obstructions arising from
  odd cycles in $G$.
\end{itemize}

These conditions rule out the vast majority of signed graphs, even among those
supported on line graphs.

\subsection{Partial characterization via forbidden signed induced subgraphs}
\label{subsec:forbidden-subgraphs}

A classical theorem of Beineke characterizes line graphs by a list of nine forbidden
\emph{unsigned} induced subgraphs \cite{Beineke1970}. Because the underlying unsigned graph
of $\mathcal A(G)$ is $L(G)$, any signed graph in the antisymmetric line graph family must,
at a minimum, have underlying graph avoiding these nine structures.

However, the canonical signing of $\mathcal A(G)$ imposes additional geometric constraints.
In particular, the cycle-parity phenomenon established in Theorem~\ref{thm:cycle-parity}
forces an infinite family of forbidden \emph{signed} induced subgraphs.

\begin{theorem}[Forbidden signed induced cycles]\label{thm:forbidden-cycles}
Let $\Sigma$ be a signed graph. If $\Sigma$ is switching-equivalent to $\mathcal A(G)$ for
some finite simple graph $G$, then for any $k\ge 4$, every induced $k$-cycle in $\Sigma$
has sign $(-1)^k$.

Equivalently, no $\mathcal A(G)$ (and hence no switching of it) contains, as an induced
signed subgraph,
\begin{enumerate}
\item an induced \emph{negative} cycle of even length $k\ge 4$, or
\item an induced \emph{positive} cycle of odd length $k\ge 5$.
\end{enumerate}
\end{theorem}

\begin{proof}
Let $C$ be an induced $k$-cycle in $\Sigma$, with $k\ge 4$. Switching preserves the sign
of every cycle, so we may assume $\Sigma=\mathcal A(G)$. Then the underlying graph
$|\Sigma|$ is $L(G)$, and $C$ corresponds to an induced $k$-cycle in $L(G)$ with vertex set
$\{e_1,\dots,e_k\}$, where each $e_i$ is an edge of $G$ and $e_i$ is adjacent in $L(G)$ to
$e_{i-1}$ and $e_{i+1}$ (indices modulo $k$).

Since $C$ is induced, $e_{i-1}$ is \emph{not} adjacent to $e_{i+1}$ in $L(G)$, hence the
edges $e_{i-1}$ and $e_{i+1}$ do not share a vertex in $G$. In particular, $e_i$ cannot
meet both $e_{i-1}$ and $e_{i+1}$ at the same endpoint (otherwise $e_{i-1}$ and $e_{i+1}$
would share that endpoint as well). Therefore $e_i$ meets $e_{i-1}$ at one endpoint and
meets $e_{i+1}$ at its other endpoint.

Define $v_i := e_i\cap e_{i+1}$ for $i=1,\dots,k$ (indices modulo $k$). The preceding
paragraph implies that each $e_i$ has endpoints $v_{i-1}$ and $v_i$, so $e_i=v_{i-1}v_i$.
Thus the edges $e_1,\dots,e_k$ form a simple $k$-cycle
\[
v_1v_2\cdots v_kv_1
\]
in $G$. The corresponding cycle on vertices $e_1,\dots,e_k$ in $\mathcal A(G)$ is exactly
the canonical lift of this $k$-cycle, and by Theorem~\ref{thm:cycle-parity} its sign is
$(-1)^k$. Hence the induced cycle $C$ has sign $(-1)^k$, as claimed.
\end{proof}

\begin{remark}
The restriction to $k\ge 4$ is necessary. An induced $3$-cycle in $L(G)$ can arise either
from a triangle in $G$ or from a $3$-star in $G$. These yield (in general) opposite signs in
$\mathcal A(G)$, cf.\ Theorem~\ref{thm:cycle-parity} and
Theorem~\ref{thm:delta3-combinatorial}. Thus both positive and negative induced triangles
can occur in antisymmetric line graphs.
\end{remark}

\subsection{Resolution of Line-Graph Ambiguities}

The rigidity of antisymmetric line graphs has an important consequence: it
eliminates the classical noninjectivity of the line graph construction.
Whitney’s exceptional pair $\{K_3,K_{1,3}\}$ arises because both graphs yield the
same underlying line graph.
However, their antisymmetric line graphs are distinguished by signed triangle
structure.

\medskip
\noindent\emph{Excluding the Whitney pair.}
It remains to rule out the exceptional possibility
$\{G,H\}=\{K_3,K_{1,3}\}$.
Although $L(K_3)\cong L(K_{1,3})\cong K_3$, their antisymmetric line graphs are
distinguished by the sign of the unique triangle cycle in the underlying $K_3$.

We compute this sign explicitly.
For a signed graph, the sign of a cycle is the product of the signs of its
edges, and this is invariant under switching.
In particular, for $\mathcal{A}(K_3)$ and $\mathcal{A}(K_{1,3})$ (both supported
on $K_3$), the sign of the unique $3$--cycle is a switching invariant.

\smallskip
\noindent\textbf{Case 1: $G=K_3$.}
Choose a reference orientation of the three edges forming a directed $3$--cycle.
Then at each vertex of $K_3$, the two incident edges have opposite incidence
signs (one enters and one leaves).
Equivalently, for each adjacent pair of edges $e\sim f$ in $L(K_3)$, the entry
$(D^{\mathsf T} D)_{ef}$ equals $-1$, so the corresponding signed adjacency in
$\mathcal{A}(K_3)$ is negative.\footnote{Here we use
$A_{\mathcal{A}(G)}=D^{\mathsf T} D-2I$ from Definition~\ref{def:ALG}; off--diagonal entries
coincide with those of $D^{\mathsf T} D$ whenever two edges share a vertex.}
Thus all three edges of $\mathcal{A}(K_3)$ are negative, and the unique triangle
cycle has sign $(-1)\cdot(-1)\cdot(-1)=-1$.

\smallskip
\noindent\textbf{Case 2: $H=K_{1,3}$.}
Let $v$ be the central vertex and orient all three edges away from $v$.
Then any pair of distinct edges incident to $v$ has the same incidence sign at
$v$ (both leave $v$), so for each adjacent pair $e\sim f$ in $L(K_{1,3})$ we have
$(D^{\mathsf T} D)_{ef}=+1$.
Hence all three edges of $\mathcal{A}(K_{1,3})$ are positive, and the unique
triangle cycle has sign $(+1)\cdot(+1)\cdot(+1)=+1$.

Therefore the signed $3$--cycle in $\mathcal{A}(K_3)$ is negative, while the
signed $3$--cycle in $\mathcal{A}(K_{1,3})$ is positive.
Since cycle signs are switching invariants, $\mathcal{A}(K_3)$ and
$\mathcal{A}(K_{1,3})$ are not switching equivalent, excluding the Whitney
exceptional pair.

\begin{theorem}[Reconstruction from the switching class]
\label{thm:reconstruction}
Let $G$ be a graph with no isolated vertices. Then the switching class of $\mathcal A(G)$
determines $G$ up to isomorphism. In general, the switching class of $\mathcal A(G)$
determines $G$ up to isomorphism modulo isolated vertices.
\end{theorem}

\begin{proof}
Switching does not change the underlying unsigned graph, hence the switching class of $\mathcal A(G)$
determines $L(G)$.
If $G$ is connected, Whitney's theorem (Theorem~\ref{thm:whitney}) implies that $L(G)$ determines $G$
up to isomorphism, except for the exceptional pair $\{K_3,K_{1,3}\}$.
The computation above shows that
$\mathcal A(K_3)$ and $\mathcal A(K_{1,3})$ are not switching equivalent, since the unique triangle
in the common underlying graph $K_3$ has opposite cycle sign in the two cases.
Thus the exceptional pair is excluded, and $G$ is determined.

For a general graph $G$, the line graph $L(G)$ determines $G$ up to isomorphism except for isolated
vertices (which are invisible to $L(G)$). The same therefore holds for the switching class of $\mathcal A(G)$.
\end{proof}

From the perspective of signed graph theory, this shows that antisymmetric line
graphs refine the line graph construction just enough to restore injectivity,
without introducing arbitrary choices or external parameters. Every graph gives rise to a signed graph in a functorial way, and classical signed
invariants—such as balance, frustration index, and signed spectra
\cite{Zaslavsky1998,BronskiDeVille2014}—therefore induce invariants of ordinary graphs when applied to $\mathcal{A}(G)$.

This perspective complements classical invariants based on adjacency,
spectrum, or expansion by exposing a layer of local structural information
that is invisible to unsigned constructions such as the line graph.

\section{The Signed Triangle Imbalance Invariant}
\label{sec:triangle}

One of the principal advantages of the antisymmetric line graph is that it
supports canonical signed-graph invariants that have no analogue at the level of
the ordinary line graph.

In this section we introduce the simplest such invariant, the \emph{signed
triangle imbalance}, and explain its combinatorial meaning.

\subsection{Definition of the Triangle Imbalance}

Let $\mathcal{A}(G)$ be the antisymmetric line graph of a graph $G$.
Since the underlying unsigned graph of $\mathcal{A}(G)$ is the line graph
$L(G)$, its triangles correspond to triples of edges in $G$ that are pairwise
adjacent.
Each such triangle carries a sign given by the product of the signs of its three
edges in $\mathcal{A}(G)$.

\begin{definition}[Signed triangle imbalance]
Let $t^+$ and $t^-$ denote the number of positive and negative triangles in
$\mathcal{A}(G)$, respectively.
The \emph{signed triangle imbalance} of $G$ is defined by
\[
\Delta_3(G) \;:=\; t^+ - t^- .
\]
\end{definition}

Since switching operations preserve the sign of every cycle, $\Delta_3(G)$ is
well-defined on the switching class of $\mathcal{A}(G)$ and therefore defines a
canonical invariant of the base graph $G$.

\subsection{Combinatorial interpretation}

Triangles in the line graph correspond to triples of edges in $G$ that are
pairwise adjacent, hence to one of two configurations: either the three edges
are incident to a common vertex (a $3$--edge star), or they form a triangle in
$G$.

\begin{theorem}[Tripods versus triangles]\label{thm:delta3-combinatorial}
Let $G$ be a finite simple graph.  Let $\Delta_3(G)=t^+-t^-$ be the signed
triangle imbalance of $\mathcal A(G)$.
Then every triangle of $\mathcal A(G)$ is positive if and only if it arises from
three edges incident to a common vertex of $G$, and every triangle of
$\mathcal A(G)$ is negative if and only if it arises from a $3$--cycle in $G$.
Consequently,
\[
t^+(G)=\sum_{v\in V(G)} \binom{d(v)}{3},\qquad
t^-(G)=\#\triangle(G),
\]
and hence
\begin{equation}\label{eq:delta3-degree-triangle}
\Delta_3(G)=\sum_{v\in V(G)} \binom{d(v)}{3}\;-\;\#\triangle(G).
\end{equation}
In particular, the total number of triangles in the line graph satisfies
\[
T(G)=t^+(G)+t^-(G)=\sum_{v\in V(G)} \binom{d(v)}{3}\;+\;\#\triangle(G).
\]
\end{theorem}

\begin{proof}
A triangle in $L(G)$ corresponds to three edges of $G$ that are pairwise adjacent.
Such a triple either (i) is incident to a common vertex $v$ (contributing
$\binom{d(v)}{3}$ choices at each $v$), or (ii) forms a $3$--cycle in $G$
(contributing $\#\triangle(G)$ choices); these are the only possibilities.

For case (i), after switching if necessary, we may assume the common vertex $v$
is a tail (or head) for all three edges, so each adjacent pair has the same
incidence at $v$, giving sign $+1$ on each of the three edges of the triangle in
$\mathcal A(G)$ and hence a positive triangle.
For case (ii), the sign of a triangle in $\mathcal A(G)$ is switching-invariant,
so it suffices to compute it for a convenient reference orientation.
Orient the $3$--cycle cyclically.
At each vertex of the cycle, the two incident edges have opposite incidence
(one enters and one leaves), hence each of the three adjacencies in the induced
triangle of $\mathcal A(G)$ has sign $-1$, and the product of the three signs is
$(-1)^3=-1$.  Therefore every triangle arising from a $3$--cycle in $G$ is negative.

The stated formulas for $t^+$, $t^-$, $\Delta_3$, and $T$ follow immediately.
\end{proof}

\begin{remark}\label{rem:delta3-limits}
Equation~\eqref{eq:delta3-degree-triangle} shows that $\Delta_3(G)$ depends only
on the degree sequence of $G$ and the number of triangles in $G$.  Thus $\Delta_3$
is a genuine signed invariant, but it is intentionally low-order and
cannot be expected to separate graphs that already agree on these basic counts.
\end{remark}

\begin{theorem}[Cycle parity in $\mathcal A(G)$]
\label{thm:cycle-parity}
Let $C \subseteq G$ be a simple cycle of length $k \ge 3$.
Then the corresponding $k$-cycle in the line graph $L(G)$
has sign $(-1)^k$ in $\mathcal A(G)$.
In particular, odd cycles in $G$ lift to negative cycles
in $\mathcal A(G)$, and even cycles lift to positive cycles.
\end{theorem}

\begin{proof}
Write $C=v_1v_2\cdots v_kv_1$ and let $e_i=v_iv_{i+1}$ (indices mod $k$) be its edges.
These $k$ edges are pairwise distinct, and consecutive edges share exactly one endpoint,
so $e_1,e_2,\dots,e_k$ form a simple $k$--cycle in the line graph $L(G)$.

The sign of a cycle is switching-invariant, so it suffices to compute it for a
convenient reference orientation of the edges of $C$.
Orient $C$ cyclically.  At each vertex of $C$, the two incident cycle edges have
opposite incidence (one enters and one leaves), hence each of the $k$ adjacencies
along the induced $k$--cycle in $\mathcal A(G)$ has sign $-1$.
Therefore the cycle sign is $(-1)^k$.
\end{proof}

\begin{corollary}[Balance and bipartiteness]
\label{cor:balanced-iff-bipartite}
$\mathcal A(G)$ is balanced if and only if $G$ is bipartite.
\end{corollary}

\begin{proof}
Suppose $G$ is bipartite with parts $A\sqcup B$.
Orient every edge of $G$ from $A$ to $B$.
If two edges share a vertex $v\in A$ then both \emph{leave} $v$, and if they share
a vertex $v\in B$ then both \emph{enter} $v$.
Thus every adjacency in $\mathcal A(G)$ has sign $+1$ in this reference orientation,
so $\mathcal A(G)$ is (trivially) balanced.

Conversely, if $G$ contains an odd cycle of length $k$,
then the corresponding cycle in $\mathcal A(G)$
has sign $(-1)^k=-1$ by Theorem~\ref{thm:cycle-parity}.
Hence $\mathcal A(G)$ contains a negative cycle and is therefore unbalanced.
\end{proof}

\begin{proposition}[Bipartite collapse]\label{prop:bipartite-collapse}
If $G$ is bipartite, then $\mathcal A(G)$ is switching-equivalent to the ordinary line graph $L(G)$ with all edges positive. In particular, every signed triangle in $\mathcal A(G)$ is positive, so
\[
\Delta_3(G)=t^+(G)=\#\triangle(L(G))=\sum_{v\in V(G)}\binom{d(v)}{3}.
\]
\end{proposition}

\begin{proof}
Let $G$ be bipartite with parts $A\sqcup B$ and orient every edge from $A$ to $B$.
If two edges of $G$ share a vertex $v\in A$ then both \emph{leave} $v$, and if they share a
vertex $v\in B$ then both \emph{enter} $v$.
Hence every adjacency in $\mathcal A(G)$ has sign $+1$ in this reference orientation, so
$\mathcal A(G)$ is switching-equivalent to the all-positive signing of $L(G)$.
In particular, every triangle in $\mathcal A(G)$ is positive, and thus $\Delta_3(G)=t^+(G)=\#\triangle(L(G))$.
For bipartite $G$, every triangle of $L(G)$ corresponds to three edges incident to a common vertex of $G$,
so $\#\triangle(L(G))=\sum_v \binom{d(v)}{3}$.
\end{proof}

\subsection{Normalization and Comparison with Classical Invariants}

To compare graphs of different sizes, it is often useful to consider a normalized
form of the triangle imbalance.
Let $T(G)$ denote the total number of triangles in $L(G)$.
We define the normalized imbalance
\[
\tau_3(G) \;:=\; \frac{\Delta_3(G)}{T(G)},
\]
whenever $T(G)>0$.

The invariant $\tau_3(G)$ should be contrasted with classical graph invariants
such as triangle density, clustering coefficient, or spectral gap
\cite{Chung1997,West2001}. 
While these invariants capture global density or expansion properties,
$\tau_3(G)$ captures a purely local signed coherence phenomenon that is invisible
to unsigned constructions.

From the perspective of signed graph theory, $\Delta_3(G)$ may also be viewed as
a low-order instance of more general cycle-balance statistics studied in the
theory of signed graphs and gain graphs \cite{Zaslavsky1982,Zaslavsky1998}.
The novelty here lies in the fact that the signing is canonical and derived from
an unsigned graph, rather than imposed externally.

Because of this, any invariant of signed graphs that is stable under switching gives rise to an invariant of ordinary graphs when applied to $\mathcal{A}(G)$.
The triangle imbalance provides the simplest nontrivial example of this
principle. The resulting invariant is closely related in spirit to classical notions of frustration and imbalance in signed graphs.

$\Delta_3(G)$ is not intended as a universal descriptor of graph
structure, nor does it subsume classical notions such as expansion or
Hamiltonicity.
Rather, it exposes a layer of local structural information—namely, the signed
coherence of edge adjacencies—that is erased by the ordinary line graph.
As such, it illustrates the broader potential of antisymmetric line graphs as a
bridge between graph theory and signed graph invariants.

\section{Computational illustrations}
\label{sec:empirical}

\paragraph{Code availability.}
Python code reproducing the computations in this section is available at
\url{https://github.com/hsbal/antisymmetric-line-graph}.

The purpose of this section is only illustrative.
We record a small sample showing that signed invariants of $\mathcal A(G)$ can
separate some graphs not distinguished by unsigned line-graph data, while also
making clear that the low-order invariant $\Delta_3$ is limited.

Let $S=A_{\mathcal A(G)}$ and $U:=|S|=A_{L(G)}$.
We use the exact trace formulas
\[
\Delta_3(G)=\frac{\operatorname{tr}(S^3)}{6},
\qquad
T(G)=\frac{\operatorname{tr}(U^3)}{6},
\]
together with
\[
\tau_3(G)=\frac{\Delta_3(G)}{T(G)}
\quad (T(G)>0),
\qquad
\mathrm{Tr}(G)=\frac{\operatorname{tr}(S^4)}{|E(G)|^2},
\]
and the signed spectral data of $S$.

\medskip
\noindent\textbf{Example 1 (separation by signed spectrum).}
Let
\[
\begin{aligned}
E(G_1)&=\{(0,1),(0,2),(0,4),(0,5),(1,2),(1,5),(2,3),(2,5),(2,6),(3,4)\},\\
E(G_2)&=\{(0,1),(0,2),(0,3),(0,4),(0,5),(1,2),(2,3),(2,5),(3,4),(4,6)\}.
\end{aligned}
\]
These graphs are non-isomorphic, have the same degree sequence
$(5,4,3,3,2,2,1)$, and have cospectral line graphs.
However,
\[
\operatorname{inertia}(S_{G_1})=(4,6,0),\qquad
\operatorname{inertia}(S_{G_2})=(4,5,1),
\]
so the signed spectra of $\mathcal A(G_1)$ and $\mathcal A(G_2)$ differ.
Here
\[
T=20,\qquad \Delta_3=12,\qquad \tau_3=0.6,
\]
for both graphs, so the separation is genuinely spectral rather than low-order.

\medskip
\noindent\textbf{Example 2 (failure of signed spectral separation).}
Let
\[
\begin{aligned}
E(H_1)&=\{(0,1),(0,2),(0,5),(1,4),(2,3),(3,5),(3,6)\},\\
E(H_2)&=\{(0,3),(0,4),(1,2),(1,3),(1,4),(2,5),(2,6)\}.
\end{aligned}
\]
These graphs are non-isomorphic, have the same degree sequence
$(3,3,2,2,2,1,1)$, and have cospectral line graphs.
In this case the signed spectra also coincide:
\[
\operatorname{inertia}(S_{H_1})=\operatorname{inertia}(S_{H_2})=(3,3,1),
\]
and again
\[
T=2,\qquad \Delta_3=2,\qquad \tau_3=1.
\]

\medskip
\noindent\textbf{Small atlas sample.}
Among the non-bipartite connected graphs with no isolated vertices on at most
seven vertices in the NetworkX graph atlas, there are eight pairs sharing the
same number of vertices and edges, the same degree sequence, and the same
adjacency spectrum of the line graph.
Within this sample, $\Delta_3$ separates none of the eight pairs, while the
signed spectrum of $\mathcal A(G)$ separates two.
Thus the signed spectrum is strictly stronger than these low-order unsigned data,
but it is not complete.

These examples should be read only as illustrations.
They support two modest conclusions: first, the canonical signed refinement can
detect structure invisible to the ordinary line graph; second, the low-order
triangle imbalance is inherently limited, as already indicated by
Remark~\ref{rem:delta3-limits}.

\section{Frustration index, Max-Cut defect, and cubic exactness}\label{sec:frustration-maxcut}

A central switching-invariant of a signed graph $\Sigma$ is its \emph{frustration index}
$\ell(\Sigma)$, defined as the minimum number of edges whose deletion makes $\Sigma$
balanced. Equivalently, by Harary's switching characterization, it is the minimum possible
number of negative edges among all switchings of $\Sigma$.
We refer to \cite{Zaslavsky1982,ArefMasonWilson2016,ArefWilson2020} for background, algorithmic
formulations, and computational aspects.

Applied to $\Sigma=\mathcal A(G)$, the quantity $\ell(\mathcal A(G))$ becomes an integer invariant of $G$ measuring how far the antisymmetric line graph is from being switching-equivalent to the all-positive line graph. Since $\mathcal A(G)$ is balanced if and only if $G$ is bipartite (Corollary~\ref{cor:balanced-iff-bipartite}), $\ell(\mathcal A(G))$ may be viewed as a ``distance from bipartiteness'' measured through the signed incidence refinement.

The edge-space viewpoint introduced in Section~\ref{sec:laplacian-geometry}
shows that the same signed operator also carries the transported Laplacian modes
and the canonical cut/cycle decomposition of edge space.
The results of the present section are of a different nature: they use the
switching class of $\mathcal A(G)$ to recover discrete graph-theoretic measures
of non-bipartiteness.

\medskip

\begin{lemma}[Orientation characterization]\label{lem:orientation-frustration}
For any finite simple graph $G$, one has
\[
\ell(\mathcal A(G)) \;=\; \min_{\vec{G}} \;\sum_{v\in V(G)} d^+_{\vec G}(v)\,d^-_{\vec G}(v),
\]
where the minimum is taken over all reference orientations $\vec G$ of $G$.
\end{lemma}

\begin{proof}
Fix an orientation $\vec G$ of $G$.
An edge of $\mathcal A(G)$ corresponds to a pair of adjacent edges $e,f$ of $G$
sharing a unique vertex $v$.
By the defining sign rule for $\mathcal A(G)$, this signed edge is negative if and
only if $e$ and $f$ have opposite incidence at $v$, i.e.\ one enters $v$ and the
other leaves.
Equivalently, negative edges of $\mathcal A(G)$ are exactly directed $2$-paths in
the orientation $\vec G$.

The number of such pairs through a fixed vertex $v$ is exactly
$d^+_{\vec G}(v)d^-_{\vec G}(v)$, and summing over $v$ counts each signed
adjacency once.
Switching $\mathcal A(G)$ corresponds to reorienting individual edges of $G$
(Proposition~\ref{prop:switching}), so minimizing the number of negative edges
over the switching class is equivalent to minimizing over orientations of $G$.
\end{proof}
\medskip

Lemma~\ref{lem:orientation-frustration} admits a direct combinatorial
interpretation: the quantity $d^+_{\vec G}(v)d^-_{\vec G}(v)$ is exactly the
number of directed $2$-paths through $v$ in the orientation $\vec G$.
Thus $\ell(\mathcal A(G))$ measures the minimum possible number of transitive
local turns over all orientations of $G$.

It is natural to compare $\ell(\mathcal A(G))$ with the \emph{Max-Cut defect}
\[
\mathrm{def}(G):=|E(G)|-\mathrm{MaxCut}(G),
\]
which is the minimum number of edges whose removal makes $G$ bipartite
(equivalently, the minimum number of edges violating a maximum bipartition).

\begin{theorem}[Lower bound via Max-Cut defect]\label{thm:frustration-lower}
For any finite simple graph $G$, one has
\[
\ell(\mathcal A(G)) \;\ge\; \mathrm{def}(G).
\]
\end{theorem}

\begin{proof}
Let $\vec G$ be an orientation achieving Lemma~\ref{lem:orientation-frustration},
so that
\[
\ell(\mathcal A(G))=\sum_v d^+(v)d^-(v).
\]
For each vertex $v$, set
\[
m(v):=\min(d^+(v),d^-(v)).
\]
Since $d^+(v),d^-(v)\in\mathbb Z_{\ge 0}$, one has
\[
m(v)\le d^+(v)d^-(v)
\]
for every $v$: this is trivial if $\min(d^+(v),d^-(v))=0$, and otherwise both
$d^+(v)$ and $d^-(v)$ are at least $1$, so
$\min(d^+(v),d^-(v))\le d^+(v)d^-(v)$.

For each vertex $v$, choose a set $F_v$ consisting of exactly the $m(v)$ edges
incident to $v$ that lie in the minority direction at $v$, and let
\[
F:=\bigcup_{v\in V(G)}F_v.
\]
Since each edge of $F$ lies in the minority direction at at least one endpoint,
each edge of $F$ contributes at least one unit to \(\sum_v m(v)\) (and possibly
two, if it lies in the minority direction at both endpoints). Therefore
\[
|F|\le \sum_v m(v)\le \sum_v d^+(v)d^-(v)=\ell(\mathcal A(G)).
\]

In the remaining oriented graph $H:=G\setminus F$, every vertex has all its
remaining incident edges oriented in the same direction, so every vertex of $H$
is either a source or a sink.
Hence $H$ is bipartite, with bipartition given by its sources and sinks.
Therefore
\[
\mathrm{def}(G)\le |F|\le \ell(\mathcal A(G)),
\]
as claimed.
\end{proof}

\begin{proposition}[Minimum directed $2$-path formulation]
\label{prop:min-transit}
For any finite simple graph $G$, the edge-frustration index $\ell(\mathcal A(G))$
is exactly the minimum number of directed $2$-paths over all orientations of $G$.
Equivalently,
\[
\ell(\mathcal A(G))
=
\min_{\vec G}
\#\bigl\{(e_1,e_2,v): e_1,e_2\in E(G)\ \text{meet at}\ v,\ e_1\ \text{enters}\ v,\ e_2\ \text{leaves}\ v\bigr\}.
\]
\end{proposition}

\begin{proof}
For a fixed orientation $\vec G$, each negative edge of $\mathcal A(G)$ corresponds
to a pair of adjacent edges of $G$ sharing a vertex $v$ with opposite incidence at
$v$, i.e.\ to a directed $2$-path through $v$.
By Lemma~\ref{lem:orientation-frustration}, the number of such pairs at $v$ is
exactly $d^+(v)d^-(v)$, and summing over $v$ counts all negative edges of
$\mathcal A(G)$.
Minimizing over orientations gives the claim.
\end{proof}

\begin{proposition}[Odd-cycle packing lower bound]
\label{prop:odd-packing-lower}
Let $\nu_{\mathrm{odd}}(G)$ denote the maximum number of edge-disjoint odd cycles
in $G$.
Then
\[
\ell(\mathcal A(G))\ge \nu_{\mathrm{odd}}(G).
\]
In particular, if $\nu_\triangle(G)$ denotes the maximum number of edge-disjoint
triangles in $G$, then
\[
\ell(\mathcal A(G))\ge \nu_\triangle(G).
\]
\end{proposition}

\begin{proof}
Every odd cycle in $G$ lifts to a negative cycle in $\mathcal A(G)$ by
Theorem~\ref{thm:cycle-parity}.
If a family of odd cycles in $G$ is edge-disjoint, then the corresponding lifted
cycles in $\mathcal A(G)$ are vertex-disjoint, since vertices of $\mathcal A(G)$
are edges of $G$.
Any edge-deletion set in $\mathcal A(G)$ that balances the signed graph must meet
every negative cycle, and an edge of $\mathcal A(G)$ can meet at most one member
of a vertex-disjoint family of cycles.
Therefore at least one deleted edge is needed for each cycle in such a family,
and hence
\[
\ell(\mathcal A(G))\ge \nu_{\mathrm{odd}}(G).
\]
The triangle-packing bound is the special case obtained by restricting to odd
cycles of length $3$.
\end{proof}

The preceding proposition gives a purely combinatorial reading of
$\ell(\mathcal A(G))$: it is the minimum number of directed $2$-paths that must
remain in any orientation of $G$.
The odd-cycle packing bound shows that this turn-level obstruction already
dominates a natural packing parameter of odd cycles.
One may also bound $\ell(\mathcal A(G))$ from above in terms of a maximum cut and
the degrees of its defect edges.

\begin{definition}[Vertex frustration number]\label{def:vertex-frustration}
Let $\Sigma$ be a signed graph. Its \emph{vertex frustration number} (also called the
\emph{frustration number}) is
\[
vf(\Sigma)\;:=\;\min\{|S|: S\subseteq V(\Sigma)\ \text{and}\ \Sigma-S\ \text{is balanced}\}.
\]
\end{definition}

\begin{theorem}[Vertex frustration equals Max-Cut defect]\label{thm:vf-equals-def}
For any finite simple graph $G$, one has
\[
vf(\mathcal A(G)) \;=\; \mathrm{def}(G).
\]
\end{theorem}

\begin{proof}
A vertex of $\mathcal A(G)$ is an edge of $G$.
Thus, for any set $S\subseteq V(\mathcal A(G))=E(G)$, deleting $S$ from
$\mathcal A(G)$ removes exactly the rows and columns indexed by those edges from
the matrix $D^{\mathsf T}D-2I$.
But this is precisely the matrix obtained from the incidence matrix of the
edge-deleted graph $G\setminus S$.
Hence
\[
\mathcal A(G)-S \;\cong\; \mathcal A(G\setminus S).
\]

By Corollary~\ref{cor:balanced-iff-bipartite}, $\mathcal A(H)$ is balanced if
and only if $H$ is bipartite.
Therefore
\[
\mathcal A(G)-S\ \text{is balanced}
\quad\Longleftrightarrow\quad
\mathcal A(G\setminus S)\ \text{is balanced}
\quad\Longleftrightarrow\quad
G\setminus S\ \text{is bipartite}.
\]
Consequently,
\[
vf(\mathcal A(G))
=
\min\{|S|:S\subseteq E(G)\ \text{and}\ G\setminus S\ \text{is bipartite}\}.
\]
The right-hand side is exactly the Max-Cut defect $\mathrm{def}(G)$, since
\[
\mathrm{def}(G)=|E(G)|-\mathrm{MaxCut}(G)
\]
equals the minimum number of edges that must be deleted to make $G$ bipartite.
Hence
\[
vf(\mathcal A(G))=\mathrm{def}(G).
\]
\end{proof}

\begin{theorem}[Upper bound via defect amplification]\label{thm:frustration-upper}
Let $V(G)=A\sqcup B$ be a maximum cut, and let $E_{\mathrm{def}}$ be the set of defect edges
(those with both endpoints in $A$ or both in $B$), so that $|E_{\mathrm{def}}|=\mathrm{def}(G)$.
Then
\[
\ell(\mathcal A(G)) \;\le\; \sum_{uv\in E_{\mathrm{def}}} \min(d(u)-1,d(v)-1).
\]
In particular,
\[
\ell(\mathcal A(G))\le (\Delta(G)-1)\,\mathrm{def}(G).
\]
\end{theorem}

\begin{proof}
Orient all cut edges from $A$ to $B$. For each defect edge $uv\subseteq A$, orient it toward
the endpoint of smaller degree; for each defect edge $uv\subseteq B$, orient it away from the
endpoint of smaller degree. Let $k^-(v)$ denote the number of defect edges entering $v\in A$.
Then $d^-(v)=k^-(v)$ and $d^+(v)=d(v)-k^-(v)$, so
\[
d^+(v)d^-(v)=k^-(v)\bigl(d(v)-k^-(v)\bigr)\le k^-(v)\,(d(v)-1).
\]
Summing over $v\in A$, each defect edge $uv\subseteq A$ oriented as $u\to v$ contributes
exactly $1$ to $k^-(v)$ and $0$ to $k^-(u)$, hence the total contribution from $A$ is at most
$\sum_{uv\subseteq A}\min(d(u)-1,d(v)-1)$. The argument for $B$ is symmetric (using outgoing
defect edges). Summing both sides yields the claimed bound, and the final inequality follows
from $\min(d(u)-1,d(v)-1)\le \Delta(G)-1$.
\end{proof}

\medskip

The sandwich
\[
\mathrm{def}(G)\ \le\ \ell(\mathcal A(G))\ \le\ (\Delta(G)-1)\,\mathrm{def}(G)
\]
shows that, for graphs of bounded maximum degree, $\ell(\mathcal A(G))$ and the Max-Cut
defect are within a constant factor of one another. In particular, this ties the frustration
index of $\mathcal A(G)$ to a classical NP-hard optimization objective.

A stronger statement holds in the cubic case: the factor-$\Delta(G)-1$ slack collapses,
and $\ell(\mathcal A(G))$ becomes \emph{exactly} twice the odd cycle transversal number.

\begin{theorem}[Exact frustration on cubic graphs]\label{thm:cubic-oct}
Let $G$ be a cubic (3-regular) graph, and let $\mathrm{oct}(G)$ denote the \emph{odd cycle
transversal number} of $G$, i.e.\ the minimum size of a set $S\subseteq V(G)$ such that the
induced subgraph $G[V(G)\setminus S]$ is bipartite. Then
\[
\ell(\mathcal A(G)) \;=\; 2\,\mathrm{oct}(G).
\]
\end{theorem}

\begin{proof}
By Lemma~\ref{lem:orientation-frustration},
\[
\ell(\mathcal A(G))=\min_{\vec G}\ \sum_{v\in V(G)} d^+(v)\,d^-(v).
\]
Assume $G$ is cubic. Then $d^+(v)+d^-(v)=3$ for all $v$, hence
\[
d^+(v)d^-(v)\in\{0,2\},
\]
with $d^+(v)d^-(v)=0$ if and only if $v$ is a pure source or pure sink in $\vec G$, and
$d^+(v)d^-(v)=2$ otherwise. Let
\[
\mathrm{Mix}(\vec G):=\{v\in V(G): 0<d^+(v)<3\}
\]
be the set of mixed vertices. Then for any orientation $\vec G$,
\[
\sum_{v\in V(G)} d^+(v)\,d^-(v)=2\,|\mathrm{Mix}(\vec G)|,
\]
so
\begin{equation}\label{eq:cubic-mix}
\ell(\mathcal A(G))=2\cdot \min_{\vec G}|\mathrm{Mix}(\vec G)|.
\end{equation}

Fix an orientation $\vec G$ and set $S:=\mathrm{Mix}(\vec G)$ and $V_0:=V(G)\setminus S$.
Partition $V_0=A\sqcup B$ where $A$ consists of pure sources and $B$ of pure sinks.
There is no edge inside $A$ (an edge between two sources would force one endpoint to have
an incoming edge) and similarly no edge inside $B$. Hence $G[V_0]$ is bipartite, so $S$ is
an odd cycle transversal. Thus $|S|\ge \mathrm{oct}(G)$, and \eqref{eq:cubic-mix} yields
$\ell(\mathcal A(G))\ge 2\,\mathrm{oct}(G)$.

Conversely, let $S\subseteq V(G)$ be an odd cycle transversal of size $\mathrm{oct}(G)$ and
write $V_0:=V(G)\setminus S$. Choose a bipartition $V_0=A\sqcup B$ of the induced subgraph
$G[V_0]$. Define an orientation $\vec G$ of all edges of $G$ by orienting every edge
incident to a vertex of $A$ away from that vertex, and every edge incident to a vertex of
$B$ towards that vertex; orient edges with both endpoints in $S$ arbitrarily. This is
well-defined: if $uv$ has $u\in A$ then we set $u\to v$, and if $uv$ has $u\in B$ then we
set $v\to u$, and for edges $uv$ with $u\in A$, $v\in B$ the two prescriptions agree.

With this orientation, every vertex in $A$ is a pure source and every vertex in $B$ is a
pure sink, hence $\mathrm{Mix}(\vec G)\subseteq S$ and $|\mathrm{Mix}(\vec G)|\le|S|$.
Applying \eqref{eq:cubic-mix} gives $\ell(\mathcal A(G))\le 2|S|=2\,\mathrm{oct}(G)$.
Combining the two inequalities proves the theorem.
\end{proof}

\begin{remark}[Mechanism in the cubic case]
The proof shows more precisely that, for a cubic graph $G$,
\[
\min_{\vec G}|\mathrm{Mix}(\vec G)|=\mathrm{oct}(G),
\]
where
\[
\mathrm{Mix}(\vec G):=\{v\in V(G):0<d^+(v)<3\}.
\]
Indeed, for any orientation $\vec G$, the mixed-vertex set is an odd cycle
transversal, while conversely any odd cycle transversal $S$ together with a
bipartition of $G[V(G)\setminus S]$ yields an orientation with
\[
\mathrm{Mix}(\vec G)\subseteq S.
\]
Combined with \eqref{eq:cubic-mix}, this is exactly the mechanism behind
Theorem~\ref{thm:cubic-oct}.
\end{remark}

\begin{corollary}[FPT transfer on cubic graphs]
On cubic graphs, computing $\ell(\mathcal A(G))$ is fixed-parameter tractable when parameterized by
$k=\ell(\mathcal A(G))$ (equivalently by $k/2=\mathrm{oct}(G)$),
since Odd Cycle Transversal is fixed-parameter tractable \cite{ReedSmithVetta2004}.
\end{corollary}

\begin{corollary}[Kernelization transfer (randomized)]
On cubic graphs, computing $\ell(\mathcal A(G))$ admits a randomized polynomial kernelization
parameterized by $k=\ell(\mathcal A(G))$, via the equivalence $k/2=\mathrm{oct}(G)$
and the randomized polynomial kernel for Odd Cycle Transversal
of Kratsch--Wahlstr\"om \cite{KratschWahlstrom2012}.
\end{corollary}

\begin{corollary}[Approximation transfer]
Any polynomial-time $\alpha$-approximation algorithm for Odd Cycle Transversal on cubic graphs yields,
in polynomial time, an $\alpha$-approximation for $\ell(\mathcal A(G))$ on cubic graphs via
$\ell(\mathcal A(G))=2\,\mathrm{oct}(G)$.
\end{corollary}

\begin{corollary}[NP-hardness]\label{cor:nphard-frustration}
Computing $\ell(\mathcal A(G))$ is NP-hard, even when the input graph $G$ is restricted to be cubic.
\end{corollary}

\begin{proof}
Odd Cycle Transversal (Graph Bipartization) is NP-hard and remains NP-complete on graphs of
maximum degree $3$ \cite{ChoiNakajimaRim1989}. Theorem~\ref{thm:cubic-oct} yields a
polynomial-time many-one reduction from $\mathrm{oct}(G)$ to $\ell(\mathcal A(G))$ on cubic inputs.
\end{proof}

\medskip
\noindent
The results of this section show that the antisymmetric line graph does not introduce a parameter foreign to classical graph theory. Rather, its switching-invariant frustration measures recover, in natural graph classes, fundamental bipartiteness parameters such as the Max-Cut defect and the odd cycle transversal number. In particular, on cubic graphs the edge-frustration index $\ell(\mathcal A(G))$ coincides exactly with $2\,\mathrm{oct}(G)$, so that a canonical signed invariant encodes a central vertex-deletion parameter of graph theory. In this sense, signed graph theory here does not extend ordinary graph theory away from its classical core, but instead provides an alternative structural encoding of it.
\medskip

\paragraph{Empirical comparison.}
For computation, on a large sample of connected non-bipartite graphs from the NetworkX atlas
with at most seven vertices and with $|E(G)|\le 12$ (to allow exact computation under a
per-instance cutoff), we computed $\ell(\mathcal A(G))$ and $\mathrm{def}(G)$ for $657$ graphs
(with $24$ timeouts). The Pearson correlation between $\ell(\mathcal A(G))$ and $\mathrm{def}(G)$
was $0.852$, indicating a strong relationship. Here $\ell(\mathcal A(G))$ was computed by
exhaustive switching search on $\mathcal A(G)$ (fixing one vertex to reduce to
$2^{|V(\mathcal A(G))|-1}$ switchings), and the $24$ timeouts refer to instances where this
exact computation exceeded the per-graph cutoff.

However, $\mathrm{def}(G)$ does not determine $\ell(\mathcal A(G))$.
For instance, within the sample we observed that graphs with $\mathrm{def}(G)=3$ can have
$\ell(\mathcal A(G))$ ranging from $3$ up to $10$, and similarly $\mathrm{def}(G)=2$ allows
$\ell(\mathcal A(G))\in\{2,3,4,5,6,7\}$.
As an explicit witness, the following two non-isomorphic graphs on $7$ vertices with $9$ edges
satisfy $\mathrm{def}(G)=\mathrm{def}(H)=2$ but have different frustration indices:
\[
\ell(\mathcal A(G))=2,\qquad \ell(\mathcal A(H))=4,
\]
where
\[
E(G)=\{01,02,12,14,23,34,35,36,56\},\qquad
E(H)=\{01,02,03,04,05,12,13,23,56\}.
\]
Thus $\ell(\mathcal A(G))$ provides a refinement of the classical Max-Cut defect, capturing
additional incidence-level obstruction beyond vertex bipartition alone.

\section{Laplacian optimization and spectral bounds}
\label{sec:laplacian-optimization}

We now return to the Gram-matrix identity
\[
A_{\mathcal A(G)}+2I=D^{\mathsf T}D
\]
from Section~\ref{sec:laplacian-geometry}, but use it in a different way.
The frustration index of $\mathcal A(G)$ is not only a signed-graph invariant and
not only a turn-counting invariant from Lemma~\ref{lem:orientation-frustration};
it is also an exact Boolean optimization of the edge-space Laplacian quadratic
form.

Fix once and for all a reference orientation of $G$ and let $D$ be the resulting
oriented incidence matrix.
For a sign vector $x\in\{\pm1\}^{E(G)}$, let $\vec G_x$ denote the orientation
obtained by reversing exactly those edges $e$ for which $x_e=-1$.
Then for each vertex $v$ one has
\[
(Dx)_v=d^+_{\vec G_x}(v)-d^-_{\vec G_x}(v).
\]

\begin{proposition}[Exact imbalance identity]
\label{prop:imbalance-identity}
For every finite simple graph $G$,
\[
\ell(\mathcal A(G))
=
\frac14\sum_{v\in V(G)} d(v)^2
-
\frac14\max_{x\in\{\pm1\}^{E(G)}} \|Dx\|^2.
\]
Equivalently, if
\[
M(G):=\max_{x\in\{\pm1\}^{E(G)}} \|Dx\|^2,
\]
then
\[
\ell(\mathcal A(G))
=
\frac14\sum_{v\in V(G)} d(v)^2-\frac14 M(G).
\]
\end{proposition}

\begin{proof}
By Lemma~\ref{lem:orientation-frustration},
\[
\ell(\mathcal A(G))
=
\min_{\vec G}\sum_{v\in V(G)} d^+_{\vec G}(v)d^-_{\vec G}(v).
\]
For any orientation $\vec G$ and any vertex $v$,
\[
d^+_{\vec G}(v)+d^-_{\vec G}(v)=d(v),
\]
and hence
\[
\bigl(d^+_{\vec G}(v)-d^-_{\vec G}(v)\bigr)^2
=
d(v)^2-4\,d^+_{\vec G}(v)d^-_{\vec G}(v).
\]
Rearranging gives
\[
d^+_{\vec G}(v)d^-_{\vec G}(v)
=
\frac14\Bigl(d(v)^2-\bigl(d^+_{\vec G}(v)-d^-_{\vec G}(v)\bigr)^2\Bigr).
\]
Summing over $v$ yields
\[
\sum_{v\in V(G)} d^+_{\vec G}(v)d^-_{\vec G}(v)
=
\frac14\sum_{v\in V(G)} d(v)^2
-
\frac14\sum_{v\in V(G)} \bigl(d^+_{\vec G}(v)-d^-_{\vec G}(v)\bigr)^2.
\]
Now every orientation of $G$ is of the form $\vec G_x$ for some
$x\in\{\pm1\}^{E(G)}$, and for this orientation the vertex-imbalance vector is
exactly $Dx$.
Therefore
\[
\sum_{v\in V(G)} \bigl(d^+_{\vec G_x}(v)-d^-_{\vec G_x}(v)\bigr)^2
=
\|Dx\|^2.
\]
Taking the minimum over orientations is therefore equivalent to taking the
maximum over sign vectors, and the result follows.
\end{proof}

\begin{remark}
The quantity
\[
M(G)=\max_{x\in\{\pm1\}^{E(G)}} \|Dx\|^2
\]
may be viewed as the \emph{maximal imbalance energy} of $G$.
It is the maximum of the edge-space Laplacian quadratic form
\[
x^{\mathsf T}D^{\mathsf T}Dx
\]
over the Boolean cube, and hence measures how strongly one can polarize the
vertex imbalances of an orientation of $G$.
\end{remark}

The exact identity above immediately yields a spectral lower bound by replacing
the Boolean cube by the Euclidean sphere.

\begin{theorem}[Spectral lower bound]
\label{thm:spectral-lower}
Let
\[
0=\lambda_1\le \lambda_2\le \cdots \le \lambda_n
\]
be the Laplacian eigenvalues of $G$, so that $\lambda_n=\lambda_{\max}(L)$.
Then
\[
\ell(\mathcal A(G))
\ge
\frac{\sum_{v\in V(G)} d(v)^2-|E(G)|\,\lambda_n}{4}.
\]
\end{theorem}

\begin{proof}
By Proposition~\ref{prop:imbalance-identity},
\[
\ell(\mathcal A(G))
=
\frac14\sum_{v\in V(G)} d(v)^2
-
\frac14\max_{x\in\{\pm1\}^{E(G)}} x^{\mathsf T}D^{\mathsf T}Dx.
\]
Every $x\in\{\pm1\}^{E(G)}$ has Euclidean norm $\|x\|^2=|E(G)|$, so
\[
\max_{x\in\{\pm1\}^{E(G)}} x^{\mathsf T}D^{\mathsf T}Dx
\le
\max_{\|y\|^2=|E(G)|} y^{\mathsf T}D^{\mathsf T}Dy.
\]
Since $D^{\mathsf T}D$ and $DD^{\mathsf T}=L$ have the same nonzero eigenvalues,
the latter maximum equals $|E(G)|\,\lambda_n$.
Substituting into Proposition~\ref{prop:imbalance-identity} gives the claim.
\end{proof}

\begin{corollary}[Combined lower bound]
\label{cor:combined-lower}
For every finite simple graph $G$,
\[
\ell(\mathcal A(G))
\ge
\max\left\{
\mathrm{def}(G),\,
\frac{\sum_{v\in V(G)} d(v)^2-|E(G)|\,\lambda_n}{4}
\right\}.
\]
\end{corollary}

\begin{proof}
Combine Theorem~\ref{thm:spectral-lower} with
Theorem~\ref{thm:frustration-lower}.
\end{proof}

For regular graphs the bound takes a particularly simple form.

\begin{corollary}[Regular case]
\label{cor:regular-spectral}
If $G$ is $\Delta$-regular on $n$ vertices, then
\[
\ell(\mathcal A(G))
\ge
\frac{n\Delta}{4}\left(\Delta-\frac{\lambda_n}{2}\right).
\]
\end{corollary}

\begin{proof}
If $G$ is $\Delta$-regular, then $\sum_v d(v)^2=n\Delta^2$ and
$|E(G)|=n\Delta/2$.
Substitute these into Theorem~\ref{thm:spectral-lower}.
\end{proof}

In the cubic case, Theorem~\ref{thm:cubic-oct} converts the spectral frustration
bound into a spectral lower bound on odd cycle transversal.

\begin{corollary}[Spectral lower bound on odd cycle transversal for cubic graphs]
\label{cor:cubic-spectral-oct}
Let $G$ be a cubic graph on $n$ vertices, and let $\lambda_n$ be the largest
Laplacian eigenvalue of $G$.
Then
\[
\mathrm{oct}(G)
\ge
\frac{\sum_{v\in V(G)} d(v)^2-|E(G)|\,\lambda_n}{8}
=
\frac{3n(6-\lambda_n)}{16}.
\]
\end{corollary}

\begin{proof}
By Theorem~\ref{thm:cubic-oct},
\[
\ell(\mathcal A(G))=2\,\mathrm{oct}(G).
\]
Combining this with Theorem~\ref{thm:spectral-lower} gives
\[
\mathrm{oct}(G)
\ge
\frac{\sum_{v\in V(G)} d(v)^2-|E(G)|\,\lambda_n}{8}.
\]
For a cubic graph, $\sum_v d(v)^2=9n$ and $|E(G)|=3n/2$, so this simplifies to
\[
\mathrm{oct}(G)\ge \frac{9n-\frac{3n}{2}\lambda_n}{8}
=
\frac{3n(6-\lambda_n)}{16}.
\]
\end{proof}

\begin{proposition}[Incomparability of the defect and spectral lower bounds]
\label{prop:incomparable-bounds}
The two lower bounds in Corollary~\ref{cor:combined-lower},
\[
\mathrm{def}(G)
\qquad\text{and}\qquad
\frac{\sum_{v\in V(G)} d(v)^2-|E(G)|\,\lambda_n}{4},
\]
are incomparable.

More precisely:
\begin{enumerate}\itemsep2pt
\item[(i)] on odd cycles $C_{2k+1}$, the defect bound is exact while the
spectral bound tends to $0$ as $k\to\infty$;
\item[(ii)] on complete multipartite graphs, the spectral lower bound captures
exactly $3/4$ of the true value of $\ell(\mathcal A(G))$.
\end{enumerate}
\end{proposition}

\begin{proof}
Part~(i) is Proposition~\ref{prop:odd-cycles-spectral-fail}, and part~(ii) is
Proposition~\ref{prop:complete-multipartite}.
\end{proof}

\begin{proposition}[Odd cycles: the spectral bound is asymptotically vacuous]
\label{prop:odd-cycles-spectral-fail}
For the odd cycle $C_{2k+1}$ one has
\[
\mathrm{def}(C_{2k+1})=1=\ell(\mathcal A(C_{2k+1})).
\]
On the other hand, the spectral lower bound from
Theorem~\ref{thm:spectral-lower} is
\[
\ell(\mathcal A(C_{2k+1}))
\ge
\frac{2k+1}{2}\left(1-\cos\frac{\pi}{2k+1}\right),
\]
and hence tends to $0$ as $k\to\infty$.
\end{proposition}

\begin{proof}
Deleting a single edge from $C_{2k+1}$ makes it bipartite, so
\[
\mathrm{def}(C_{2k+1})=1.
\]
Also, by Theorem~\ref{thm:frustration-lower}, one has
\[
1=\mathrm{def}(C_{2k+1})\le \ell(\mathcal A(C_{2k+1})).
\]
On the other hand, orient the cycle cyclically. Then exactly one vertex of the
orientation contributes a directed $2$--path count of $1$, while every other
vertex contributes $0$, so Lemma~\ref{lem:orientation-frustration} gives
\[
\ell(\mathcal A(C_{2k+1}))\le 1.
\]
Hence
\[
\ell(\mathcal A(C_{2k+1}))=1.
\]

For the spectral term, the Laplacian eigenvalues of $C_n$ are
\[
2-2\cos\left(\frac{2\pi j}{n}\right),\qquad j=0,1,\dots,n-1.
\]
Thus for odd $n=2k+1$,
\[
\lambda_n(C_{2k+1})=2+2\cos\left(\frac{\pi}{2k+1}\right).
\]
Since $C_{2k+1}$ is $2$-regular on $2k+1$ vertices,
\[
\sum_{v\in V(C_{2k+1})} d(v)^2=4(2k+1),
\qquad
|E(C_{2k+1})|=2k+1.
\]
Substituting into Theorem~\ref{thm:spectral-lower} gives
\[
\ell(\mathcal A(C_{2k+1}))
\ge
\frac{4(2k+1)-(2k+1)\lambda_n(C_{2k+1})}{4}
=
\frac{2k+1}{2}\left(1-\cos\frac{\pi}{2k+1}\right).
\]
Since $1-\cos t\sim t^2/2$ as $t\to 0$, this is asymptotic to
\[
\frac{\pi^2}{4(2k+1)},
\]
which tends to $0$.
\end{proof}

\begin{proposition}[Complete multipartite graphs: exact evaluation and $3/4$ sharpness]
\label{prop:complete-multipartite}
Let
\[
G=K_{n_1,\dots,n_q}
\]
be the complete multipartite graph with part sizes $n_1,\dots,n_q$, and write
\[
n=\sum_{i=1}^q n_i.
\]
Then
\[
\ell(\mathcal A(G))
=
\sum_{1\le i<j<k\le q} n_i n_j n_k.
\]
Moreover, the spectral lower bound from Theorem~\ref{thm:spectral-lower} is
\[
\ell(\mathcal A(G))
\ge
\frac34\sum_{1\le i<j<k\le q} n_i n_j n_k.
\]
\end{proposition}

\begin{proof}
By Lemma~\ref{lem:orientation-frustration},
\[
\ell(\mathcal A(G))
=
\min_{\vec G}\sum_{v\in V(G)} d^+_{\vec G}(v)d^-_{\vec G}(v),
\]
and the right-hand side counts directed $2$--paths.

Every triangle of $G$ is obtained by choosing one vertex from each of three
distinct parts. Hence
\[
\#\{\text{triangles of }G\}
=
\sum_{1\le i<j<k\le q} n_i n_j n_k.
\]
In any orientation of a triangle, there is at least one directed $2$--path.
Therefore every orientation of $G$ has at least one directed $2$--path per
triangle, so
\[
\ell(\mathcal A(G))
\ge
\sum_{1\le i<j<k\le q} n_i n_j n_k.
\]

For the reverse inequality, order the parts
\[
V_1,\dots,V_q
\]
and orient every edge from the earlier part to the later part.
Then every triangle is transitive and contributes exactly one directed
$2$--path.
At a vertex in part $V_j$, all incoming edges come from parts $V_i$ with $i<j$
and all outgoing edges go to parts $V_k$ with $k>j$, so every directed
$2$--path through that vertex has endpoints in two distinct parts and hence lies
in a unique triangle.
Thus the total number of directed $2$--paths is exactly the number of triangles,
and therefore
\[
\ell(\mathcal A(G))
=
\sum_{1\le i<j<k\le q} n_i n_j n_k.
\]

It remains to evaluate the spectral lower bound.
For a complete multipartite graph one has $\lambda_n(G)=n$.
Also, every vertex in part $i$ has degree $n-n_i$, so
\[
\sum_{v\in V(G)} d(v)^2=\sum_{i=1}^q n_i(n-n_i)^2,
\]
and
\[
|E(G)|=\frac{n^2-\sum_{i=1}^q n_i^2}{2}.
\]
Hence Theorem~\ref{thm:spectral-lower} gives
\[
\ell(\mathcal A(G))
\ge
\frac{1}{4}\left(\sum_{i=1}^q n_i(n-n_i)^2
-\frac{n}{2}\left(n^2-\sum_{i=1}^q n_i^2\right)\right).
\]
Expanding and simplifying yields
\[
\ell(\mathcal A(G))
\ge
\frac18\left(n^3-3n\sum_{i=1}^q n_i^2+2\sum_{i=1}^q n_i^3\right).
\]
On the other hand, the elementary symmetric identity
\[
\sum_{1\le i<j<k\le q} n_i n_j n_k
=
\frac16\left(n^3-3n\sum_{i=1}^q n_i^2+2\sum_{i=1}^q n_i^3\right)
\]
shows that the spectral lower bound equals
\[
\frac34\sum_{1\le i<j<k\le q} n_i n_j n_k
=
\frac34\,\ell(\mathcal A(G)).
\]
\end{proof}

\begin{remark}[Dense overlap versus sparse obstruction]
\label{rem:dense-sparse-regimes}
Propositions~\ref{prop:odd-cycles-spectral-fail} and
\ref{prop:complete-multipartite} isolate the two extreme regimes already
suggested by the computational evidence.
On sparse near-bipartite graphs such as odd cycles, the defect bound is the
correct obstruction and the spectral term may be very weak.
On dense graphs with high local edge overlap, exemplified by complete
multipartite graphs, the spectral term is substantially stronger.
This matches the interpretation of Proposition~\ref{prop:imbalance-identity}:
the spectral lower bound is strongest when the incidence Gram matrix
$D^{\mathsf T}D$ records many overlapping local $2$--path configurations.
\end{remark}

\begin{remark}[Position in the literature]
\label{rem:spectral-literature}
Spectral approaches to non-bipartiteness are classical, but they typically take a
different form from Corollary~\ref{cor:cubic-spectral-oct}.
For example, the largest eigenvalue of the normalized Laplacian and related
bipartiteness-ratio inequalities have been studied by Bauer--Jost
\cite{BauerJost2013} and in Trevisan's spectral Max-Cut work
\cite{Trevisan2012}, while Fallat and Fan \cite{FallatFan2012} relate the least
signless-Laplacian eigenvalue to vertex and edge bipartiteness parameters.
By contrast, Corollary~\ref{cor:cubic-spectral-oct} is a lower bound on odd
cycle transversal in terms of the largest \emph{ordinary} Laplacian eigenvalue,
and it arises through the antisymmetric-line-graph identity
\[
\ell(\mathcal A(G))=2\,\mathrm{oct}(G)
\]
in the cubic case.
We are not aware of this particular route or bound appearing previously.
\end{remark}

\section{Discussion and outlook}
\label{sec:outlook}

We have studied the antisymmetric line graph $\mathcal{A}(G)$ as a canonical
signed graph attached to an ordinary graph, with an emphasis different from the
existing signed-spectral literature. The signed graph with adjacency
\[
A_{\mathcal A(G)}=D^{\mathsf T}D-2I
\]
is classical as the signed or spectral line graph of an oriented signed graph
\cite{GerminaHameedZaslavsky2011,Zaslavsky2010Matrices,BelardoPisanskiStanicZaslavsky2023}.
What is specific to the present paper is the use of this canonical switching
class as a source of invariants of the underlying unsigned graph.

From that perspective, the strongest results obtained here are the
frustration-theoretic ones. We proved the orientation formula for
$\ell(\mathcal A(G))$, the comparison with Max-Cut defect, the exact cubic
identity
\[
\ell(\mathcal A(G))=2\,\mathrm{oct}(G),
\]
the Boolean edge-space Laplacian reformulation of $\ell(\mathcal A(G))$, the
resulting spectral lower bound, and the exact complete multipartite evaluation
with $3/4$-sharpness. These results appear to be new and show that the canonical
signed line graph can be used to extract nontrivial combinatorial information
about the original unsigned graph.

A second theme, recorded here more compactly, is that the same matrix identity
\[
A_{\mathcal A(G)}+2I=D^{\mathsf T}D
\]
places the Laplacian on edge space and exposes the cut/cycle decomposition and
the distinguished $(-2)$-sector. This viewpoint is classical at the operator
level, but it remains useful in the present context because it explains why the
frustration index of $\mathcal A(G)$ admits an exact Boolean-Laplacian
formulation.

Natural next questions include a fuller structural characterization of those
signed graphs that arise as $\mathcal A(G)$, sharper spectral relaxations for
frustration-type parameters, and further development of the lift-theoretic
framework in which $\mathcal A(G)$ appears as the antisymmetric sector of
$\mathrm{HL}'_2(G)$ \cite{BalHL}.

\section*{Statements and Declarations}
\textbf{Funding.} No funding was received for this work.

\textbf{Competing Interests.} The author has no competing interests to declare.

\vspace{2em}
\noindent\textbf{Author address:} \\
Hartosh Singh Bal \\
The Caravan, Jhandewalan Extn., New Delhi 110055, India \\
\texttt{hartoshbal@gmail.com}

\end{document}